\documentclass{amsart}
\input epsf

\usepackage{amsfonts,amsthm,amsmath,amssymb,latexsym}
\usepackage{graphicx,color}
\usepackage[all]{xy}
\usepackage{labelfig}
\usepackage{epsfig}

\begin{document}

\author{Dragomir \v Sari\' c}
\thanks{This research is partially supported by National Science Foundation grant DMS 1102440.}

\address{Department of Mathematics, Queens College of CUNY,
65-30 Kissena Blvd., Flushing, NY 11367}
\email{Dragomir.Saric@qc.cuny.edu}

\address{Mathematics PhD. Program, The CUNY Graduate Center, 365 Fifth Avenue, New York, NY 10016-4309}

\theoremstyle{definition}

 \newtheorem{definition}{Definition}[section]
 \newtheorem{remark}[definition]{Remark}
 \newtheorem{example}[definition]{Example}

\newtheorem*{notation}{Notation}

\theoremstyle{plain}

 \newtheorem{proposition}[definition]{Proposition}
 \newtheorem{theorem}[definition]{Theorem}
 \newtheorem{corollary}[definition]{Corollary}
 \newtheorem{lemma}[definition]{Lemma}

\def\H{{\mathbb H}}
\def\F{{\mathcal F}}
\def\R{{\mathbb R}}
\def\Q{{\mathbb Q}}
\def\Z{{\mathbb Z}}
\def\E{{\mathcal E}}
\def\N{{\mathbb N}}
\def\X{{\mathcal X}}
\def\Y{{\mathcal Y}}
\def\C{{\mathbb C}}
\def\D{{\mathbb D}}
\def\G{{\mathcal G}}

\title[Thurston's boundary]{Thurston's boundary for Teichm\"uller spaces of infinite surfaces: the length spectrum}


\date{\today}

\maketitle

\begin{abstract}
Let $X$ be an infinite geodesically complete hyperbolic surface which can be decomposed into geodesic pairs of pants. We introduce Thurston's 
boundary to the Teichm\"uller space $T(X)$ of the surface $X$ using the length spectrum analogous to Thurston's construction for finite surfaces. Thurston's 
boundary using the length spectrum of $X$ is a ``closure'' of projective bounded measured laminations $PML_{bdd}
(X)$, and it coincides with $PML_{bdd}(X)$ when $X$ can be decomposed into a countable union of geodesic pairs of pants whose 
boundary geodesics $\{\alpha_n\}_{n\in\mathbb{N}}$ have lengths pinched between two positive constants. When a subsequence of the lengths of the boundary curves of the geodesic pairs of pants $\{\alpha_n\}_n$ converges to zero, Thurston's 
boundary using the length spectrum is strictly larger than $PML_{bdd}(X)$.\end{abstract}

\section{Introduction}

Fix a geodesically complete infinite area hyperbolic surface $X_0$. The space of all quasiconformal deformations of $X$ modulo conformal maps and homotopies is an infinite-dimensional 
manifold called the Teichm\"uller space $T(X_0)$ of $X_0$. We study the limiting 
behaviour 
of the quasiconformal deformations of $X$ when the dilatations of the quasiconfomal 
maps increase without bound using the length spectrum. Thurston \cite{Th1}, \cite{FLP} 
used the length spectrum to compactify the Teichm\"uller space of a closed surface by 
adding to it the space of projective measured laminations of the surface. Bonahon \cite{Bo} 
used geodesic currents to give an alternative description of Thurston's boundary for 
the Teichm\"uller space of a closed surface. In \cite{Sar5}, geodesic currents were used to introduce Thurston's boundary to $T(X)$.

A complete hyperbolic surface is obtained by gluing geodesic pairs of pants (with possible at most two punctures on the boundary) and by adding at most countably many funnels with closed geodesic boundary and half-planes with boundary infinite geodesics (cf. \cite{BasSar}). 
Since we consider the length spectrum, it is natural to restrict our attention to geodesically complete, infinite hyperbolic surfaces that are obtained by gluing countably many geodesic pairs of pants (cf. \cite{Shi}, \cite{BK}, \cite{ALPS}). For such surfaces, the Teichm\"uller space is completely determined by the marked length spectrum.
The study of the length spectrum properties for infinite 
surfaces is started by Shiga \cite{Shi}, and it was further developed by various 
authors(e.g. \cite{ALPS}, \cite{ALPS1}, \cite{BK} \cite{Matz}, \cite{Kin}, \cite{Sa1},...).

Denote by $[f]\in T(X_0)$ the equivalence class of a quasiconformal map $f:X_0\to X$.
Let $\mathcal{S}$ be the set of all simple closed geodesics on $X_0$.
Homotopy class of a quasiconformal map $f:X_0\to X$ induces a function from $\mathcal{S}$ to $\mathbb{R}$ which assigns to each $\alpha\in\mathcal{S}$ the length of a geodesic in $X$ that is homotopic to $f(\alpha )$.  
Thus we have an injective map
$$
\mathcal{X}:T(X_0)\to\mathbb{R}_{\geq 0}^{\mathcal{S}}.
$$
When $X_0$ is a closed hyperbolic surface then the above map is a homeomorphism onto its image if $\mathbb{R}_{\geq 0}^{\mathcal{S}}$ is equipped with the weak* topology (cf. \cite{FLP}). In the case of an infinite surface with a geodesic pants decomposition, the {\it length spectrum metric} is defined by (cf. \cite{Shi}, \cite{ALPS1})
\begin{equation*}
d_{ls}([f],[g])=\sup_{\alpha\in\mathcal{S}}\Big{|}\log \frac{l_{f(X_0)}(f(\alpha ))}{l_{g(X_0)}(g(\alpha ))}\Big{|}.
\end{equation*}
Shiga \cite{Shi}  proved that the topology induced by the length spectrum metric on $T(X_0)$ is equal to the Teichm\"uller topology when the surface $X$ has a geodesic pants decomposition with lengths of boundary geodesics (of the pants) pinched between two positive constants. Allessandrini, Liu, Papadopoulos and Su \cite{ALPS} proved that the length spectrum on $T(X_0)$ is not  complete when $X_0$ contains a sequence of simple closed geodesics whose length goes to zero. Thus the two topologies in this case are different.

We introduce a {\it normalized supremum norm} on $\mathbb{R}^{\mathcal{S}}_{\geq 0}$  by
\begin{equation*}
\| f\|_{\infty}^{norm}=\sup_{\alpha\in\mathcal{S}}\Big{|}
\frac{f(\alpha )}{l_{X_0}(\alpha )}
\Big{|}
\end{equation*}
for $f\in \mathbb{R}^{\mathcal{S}}_{\geq 0}$. The normalized supremum norm on $\mathbb{R}^{\mathcal{S}}_{\geq 0}$ makes the map $
\mathcal{X}:T(X_0)\to\mathbb{R}_{\geq 0}^{\mathcal{S}}
$ a homeomorphism onto its image (cf. Lemma \ref{prop:two_metrics}).

Analogous to the closed surface case, we projectivize $\mathcal{X}$ and obtain an injective map
$$
P\mathcal{X}:T(X_0)\to P\mathbb{R}^{\mathcal{S}}_{\geq 0}.
$$
By definition, {\it (length spectrum) Thurston's boundary} of $T(X_0)$ consists of the boundary points of the image $P\mathcal{X}(T(X_0))$ of $T(X_0)$, where $P\mathbb{R}^{\mathcal{S}}_{\geq 0}$ is given the quotient topology with respect to the normalized supremum norm on $\mathbb{R}^{\mathcal{S}}_{\geq 0}$.

\vskip .2 cm

\noindent {\bf Theorem 1.}
{\it
Let $X_0$ be an infinite area geodesically complete hyperbolic surface that has a geodesic pants decomposition with boundary geodesics of pair of pants $\{\alpha_n\}_{n\in\mathbb{N}}$. Then  (length spectrum) Thurston's boundary of $T(X_0)$ is the closure of the space of projective bounded measured laminations $PML_{bdd}(X_0)$ in $P\mathbb{R}^{\mathcal{S}}$, where $P\mathbb{R}^{\mathcal{S}}$ has the quotient topology induced by the topology on $\mathbb{R}^{\mathcal{S}}$ coming from the normalized supremum norm. 

If the lengths of $\{\alpha_n\}_{n\in\mathbb{N}}$ are pinched between two positive constants then length spectrum Thurston's boundary is equal to $PML_{bdd}(X_0)$.

If the lengths of $\{\alpha_n\}_{n\in\mathbb{N}}$ are bounded from the above and there exists a subsequence $\{\alpha_{n_k}\}$ whose lengths converge to $0$, then length spectrum Thurston's boundary is strictly larger than $PML_{bdd}(X_0)$.
}

\vskip .2 cm

In addition, Thurston's boundary of a hyperbolic surface $X_0$ whose every geodesic pants decomposition does not have an upper bound on the lengths of cuffs but that can be decomposed into bounded polygons with at most $n$ sides (introduced by Kinjo \cite{Kin}) equals $PML_{bdd}(X_0)$. On the other hand, if $X_0$ is the surface constructed by Shiga \cite{Shi} such that the length spectrum metric is incomplete, then length spectrum Thurston's boundary is strictly larger than $PML_{bdd}(X_0)$ (cf. \S \ref{sec:twosurfaces}).

\vskip .2 cm

Recall that the quasiconformal Mapping Class Group $MCG_{qc}(X_0)$ consists of all quasiconformal maps $g:X_0\to X_0$ up to homotopy (cf. \cite{GL}). The action of $MCG_{qc}(X_0)$ on the Teichm\"uller space $T(X_0)$ is given by
$
[f]\mapsto [f\circ g^{-1}]$ and it is continuous in the Teichm\" uller metric. Therefore it is also continuous in the length spectrum metric. The normalised supremum norm on $\mathbb{R}^{\mathcal{S}}_{\geq 0}$ is invariant under the change of markings of $\mathcal{S}$ because it is the supremum over all simple closed curves $\mathcal{S}$. Therefore we obtain

\vskip .2 cm

\noindent {\bf Theorem 2.} {\it The action of the quasiconformal Mapping Class Group $MCG_{qc}(X_0)$ on the Teichm\"uller space $T(X_0)$ extends to a continuous action on (length spectrum) Thurston's closure of $T(X_0)$.}

\section{Teichm\"uller spaces of geometrically infinite hyperbolic surfaces}

Let $X_0$ be a geodesically complete hyperbolic surface  whose area is infinite. The universal covering $\tilde{X}_0$ of the surface $X_0$ is isometrically identified with the hyperbolic plane $\mathbb{H}$. The boundary at infinity $\partial_{\infty}\tilde{X}_0$ is identified with the unit circle $S^1$.

The {\it Teichm\"uller space} $T(X_0)$ of the surface $X_0$ is the space of equivalence classes of all quasiconformal maps $f:X_0\to X$ where $X$ is an arbitrary complete hyperbolic surface modulo an equivalence relation. Two quasiconformal maps $f_1:X_0\to X_1$ and $f_2:X_0\to X_2$ are {\it equivalent} if there exists an isometry $I:X_1\to X_2$ such that $f_2^{-1}\circ I\circ f_1$ is homotopic to the identity under a bounded homotopy. Denote by $[f]$ the equivalence class of a quasiconformal map $f:X_0\to X$.

The {\it Teichm\"uller distance} on $T(X_0)$ is defined by
$$
d_{T}([f_1],[f_2])=\frac{1}{2}\log \inf_{g\simeq f_2\circ f_1^{-1}}K(g)
$$
where the infimum is taken over all quasiconformal maps $g$ homotopic to $f_2\circ f_1^{-1}$ and $K(g)$ is the quasiconformal constant of $g$. The {\it Teichm\"uller topology} on $T(X_0)$ is the topology induced by the Teichm\"uller distance.

\section{Measured laminations and earthquakes}

A {geodesic lamination} on a hyperbolic surface $X$ is a closed subset of $X$ that is foliated by non-intersecting complete geodesics called {\it leaves} of the lamination. Geodesic lamination on $X$ lifts to a geodesic lamination on $\mathbb{H}$ that is invariant under the action  of the covering group of $X$. A {\it stratum} of a geodesic lamination is either a leaf of the lamination or a connected component of the complement. A connected component of the complement of a geodesic lamination in $\mathbb{H}$ is isometric to a possibly infinite sided geodesic polygon whose sides are complete geodesics and possibly arcs on $S^1$. 

A {\it measured lamination} $\mu$ on $X$ is an assignment of a positive Borel measure on each arc transverse to a geodesic lamination $|\mu |$ that is invariant under homotopies relative leaves of $|\mu |$. The geodesic lamination $|\mu |$ is called the {\it support} of $\mu$. A measured lamination on $X$ lifts to a measured lamination on $\mathbb{H}$ that is invariant under the covering group of $X$.

 A {\it left earthquake} $E:X_0\to X$ with support geodesic lamination $\lambda$ is a surjective map that is isometry on each stratum of $\lambda$ such that each stratum is moved to the left relative to any other stratum. 
An earthquake of $X_0$ lifts to an earthquake of $\mathbb{H}$ where the support is the lift of the support on $X_0$ (cf. Thurston \cite{Th1}).

We give a definition of a (left) earthquake $E:\mathbb{H}\to\mathbb{H}$ with support geodesic lamination $\lambda$ on $\mathbb{H}$. A {\it left earthquake} $E:\mathbb{H}\to\mathbb{H}$ is a bijection of $\mathbb{H}$ whose restriction to any stratum of $\lambda$ is an isometry of $\mathbb{H}$; if $A$ and $B$ are two strata of $\lambda$ then
$$
(E|_A)^{-1}\circ E|_B
$$
is a hyperbolic translation whose axis weakly separates $A$ and $B$ that moves $B$ to the left as seen from $A$ (cf. Thurston \cite{Th1}).

An earthquake $E:\mathbb{H}\to\mathbb{H}$ induces a transverse measure $\mu$ to its support $\lambda$ which defines a measured lamination $\mu$ with $|\mu |=\lambda$ (cf. \cite{Th1}). 
An earthquake of $\mathbb{H}$ extends by continuity to a homeomorphism of $S^1$. Thurston's earthquake theorem states that any homeomorphism of $S^1$ can be obtained by continuous extension of a left earthquake (cf. Thurston \cite{Th1}).

Given a measured lamination $\mu$, there exists a map $E^{\mu}:\mathbb{H}\to\mathbb{H}$ whose transverse measure is $\mu$ and that satisfies all properties in the definition of an earthquake of $\mathbb{H}$ except being onto (cf. \cite{Th1}, \cite{GHL}). $E^{\mu}$ is uniquely determined by $\mu$ up to post-composition by an isometry of $\mathbb{H}$.

We define {\it Thurston's norm} of a measured lamination $\mu$ as
$$
\|\mu\|_{Th}=\sup_J i(\mu ,J)
$$
where the supremum is over all hyperbolic arcs $J$ of length $1$.

A quasiconformal map of $X$ lifts to a quasiconformal map of $\mathbb{H}$ and the later extends to a quasisymmetric map of $S^1$. 
Therefore we consider measured 
laminations whose earthquakes induces quasisymmetric maps of $S^1$. 
An earthquake $E^{\mu}$ extends by continuity to a quasisymmetric map of $S^1$ if and only if $\|\mu
\|_{Th}<\infty$ (cf. \cite{Th1}, \cite{GHL}, \cite{Sa1}, \cite{Sa2}).

Denote by $ML_{bdd}(X)$ and $ML_{bdd}(\mathbb{H})$ the space of all measured laminations 
with finite Thurston's norm on $X$ and $\mathbb{H}$, respectively. The above result gives a bijective map
$$
EM:T(X)\to ML_{bdd}(X).
$$

Note that $\| t\mu\|_{Th}=t\|\mu\|_{Th}$, for $t>0$. Then, for $\|\mu\|_{Th}<\infty$, we have that
$t\mapsto E^{t\mu}$, for $t>0$, is a path in $T(X)$ called {\it earthquake path}.

\section{Thurston's boundary for Teichm\"uller spaces of infinite surfaces using the length spectrum}

In this section we consider infinite type hyperbolic surfaces and introduce 
 ``length spectrum'' Thurston's boundary to their Teichm\"uller spaces. It turns out that length spectrum Thurston's boundary differs 
from Thurston's boundary introduced using geodesic currents (see \cite{Sar5} for the construction using geodesic currents).

Recall that $X_0$ is a geodesically complete hyperbolic surface that has a 
geodesic pants decomposition. In other words, $X_0$ is formed by gluing infinitely many geodesic pairs of pants along their boundaries.

 Let $\{\alpha_n\}_{n\in\mathbb{N}}$ be the family of cuffs (i.e. boundary 
 components) of a geodesic pants decomposition of $X_0$ as above.
 Then each $\alpha_n$ is a simple closed geodesic or a puncture. We say that $\{\alpha_n\}$ is an {\it upper-bounded geodesic pants decomposition} of $X_0$ if there exists $M>0$ such that, for each $n\in\mathbb{N}$,
 $$
 l_{X_0}(\alpha_n)\leq M
 $$
where $l_{X_0}(\alpha_n)$ is the length of $\alpha_n$ for the hyperbolic metric of $X_0$ (cf. \cite{ALPS}). Moreover, $\{\alpha_n\}$ is a {\it lower-bounded geodesic pants decomposition} of $X_0$ if there exists $m>0$ such that, for each $n\in\mathbb{N}$,
 $$
 l_{X_0}(\alpha_n)\geq m.
 $$

\subsection{General infinite surfaces}

Denote by $\mathcal{S}$ the set of all simple closed geodesics on a geodesically complete hyperbolic surface
$X_0$ equipped with a geodesic pants decomposition. Let $\mathbb{R}_{\geq 0}^{\mathcal{S}}$ be the space of non-negative functions on the set of all simple 
closed geodesics $\mathcal{S}$ of $X_0$. We define a map $\mathcal{X}$ from the Teichm\"uller space $T(X_0)$ into $\mathbb{R}_{\geq 0}^{\mathcal{S}}$, for $[f]\in T(X_0)$ and $\alpha\in\mathcal{S}$,
$$
\mathcal{X}([f])(\alpha )=l_{f(X_0)}(f(\alpha )),
$$
where $f(X_0)$ is the image hyperbolic surface under quasiconformal mapping $f$ and $l_{f(X_0)}(f(\alpha ))$ is the length of a simple closed geodesic on $f(X_0)$  homotopic to a simple closed curve $f(\alpha )$.
The map $\mathcal{X}:T(X_0)\to \mathbb{R}_{\geq 0}^{\mathcal{S}}$ is injective.

The {\it length spectrum metric} on $T(X_0)$ is given by
\begin{equation*}
d_{ls}([f_1],[f_2])=\sup_{\delta\in \mathcal{S}}\Big{\{} 
|\log\frac{l_{f_2(X_0)}(f_2(\delta ))}{l_{f_1(X_0)}(f_1(\delta ))}|\Big{\}}.
\end{equation*}
Shiga \cite{Shi} proved that 
if $X_0$ has an upper and lower bounded geodesic pants decomposition $\{\alpha_n\}_{n\in\mathbb{N}}$ then the Teichm\"uller distance induces the same topology as the length-spectrum distance on $T(X_0)$.

We introduce the {\it normalized supremum norm} on $\mathbb{R}^{\mathcal{S}}$ by
$$
\| f\|_{\infty}^{norm}=\sup_{\delta\in\mathcal{S}}\frac{|f(\delta )|}{l_{X_0}(\delta )}
$$
for all $f\in \mathbb{R}^{\mathcal{S}}$. Note that the normalized supremum norm on $ \mathbb{R}^{\mathcal{S}}$ is infinite at some points of $\mathbb{R}^{\mathcal{S}}$. We consider only the subset of $ \mathbb{R}^{\mathcal{S}}$ where the normalized supremum is finite and, for simplicity, denote it by $ \mathbb{R}^{\mathcal{S}}$.

\begin{proposition}
\label{prop:two_metrics}
The length spectrum metric on $T(X_0)$ is locally bi-Lipschitz equivalent to the  normalized supremum norm on $\mathcal{X}(T(X_0))$.
\end{proposition}

\begin{proof}
Indeed, if 
$$
\sup_{\delta\in\mathcal{S}}\Big{|}\frac{l_{f_1(X_0)}(f_1(\delta ))}{l_{X_0}(\delta )}-
\frac{l_{f_2(X_0)}(f_2(\delta ))}{l_{X_0}(\delta )}\Big{|}
<\epsilon
$$
then
$$
\sup_{\delta\in\mathcal{S}}\frac{l_{f_1(X_0)}(f_1(\delta ))}{l_{X_0}(\delta )}\Big{|}1-
\frac{l_{f_2(X_0)}(f_2(\delta ))}{l_{f_1(X_0)}(f_1(\delta ))}\Big{|}
<\epsilon.
$$
Since $f_1$ is a quasiconformal map, there exists $M>1$ such that $1/M
\leq \frac{l_{f_1(X_0)}(f_1(\delta ))}{l_{X_0}(\delta )}\leq M$ (cf. Wolpert \cite{Wol}).
The above and symmetry implies
$$
\Big{|}\frac{l_{f_2(X_0)}(f_2(\delta ))}{l_{f_1(X_0)}(f_1(\delta ))}
-1\Big{|}, \Big{|}\frac{l_{f_1(X_0)}(f_1(\delta ))}{l_{f_2(X_0)}(f_2(\delta ))}
-1\Big{|}\leq M\epsilon
$$
for all $\delta\in \mathcal{S}$,
and one direction is obtained since $|\log x|/ |x-1|$ is between two positive constants for $1/2<x<2$. The other direction is obtained by reversing the order of the above inequalities and the two metrics are locally bi-Lipschitz. 
\end{proof}

Allessandrini, Liu, Papadopoulos and Su \cite{ALPS1} proved that $T(X_0)$ is not complete in the length spectrum metric 
 when there exists a sequence of simple closed geodesics on $X_0$ whose lengths converge to $0$. 
Thus, 
$\mathcal{X}:T(X_0)\to\mathbb{R}_{\geq 0}^{\mathcal{S}}$ is not a homeomorphism onto its image for the normalized supremum norm on $\mathbb{R}_{\geq 0}^{\mathcal{S}}$ and the Teichm\"uller metric on $T(X_0)$ when $X_0$ contains a sequence of simple closed geodesics whose lengths converge to zero.

Denote by
$$
\mathcal{PX}:T(X_0)\to P\mathbb{R}^{\mathcal{S}}_{\geq 0}
$$
the map from $T(X_0)$ into the projective space $P\mathbb{R}^{\mathcal{S}}_{\geq 0}=
(\mathbb{R}^{\mathcal{S}}_{\geq 0}-\{ \bar{0}\}) /\mathbb{R}_{>0}$.
The map $\mathcal{PX}$ is injective on $T(X_0)$. The {\it length spectrum Thurston's boundary} of $T(X_0)$ is, by the definition, the space of all limit points in $P\mathbb{R}^{\mathcal{S}}_{\geq 0}$ of the set $\mathcal{PX}(T(X_0))$ for the topology induced by the normalized supremum norm (c.f. \cite{FLP} for the original Thurston's discussion on closed surfaces). 

Note that a measured lamination $\mu$ on $X_0$ represent an element in $\mathbb{R}^{\mathcal{S}}_{\geq 0}$ by the formula
$$
\mu (\alpha )=i(\mu ,\alpha )
$$
for all $\alpha\in\mathcal{S}$, where $i(\mu ,\alpha )$ is the intersection number.

\begin{proposition}
\label{prop:p1}
Let $X_0$ be a geodesically complete infinite hyperbolic surface equipped with a geodesic pants 
decomposition. Then length spectrum Thurston's boundary of $T(X_0)$ contains 
the space of projective bounded measured lamination $PML_{bdd}(X_0)$ and it equals the closure of $PML_{bdd}(X_0)$ for the topology on $P\mathbb{R}_{\geq 0}^{\mathcal{S}}$ induced by the normalized supremum norm.
\end{proposition}

\begin{proof}
Let $\mu\in ML_{bdd}(X_0)$ be a bounded measured lamination on 
$X_0$. Denote by $E^{t\mu}$, for $t>0$, an earthquake path with the 
earthquake measure $t\mu$. Then $t\mapsto E^{t\mu}(X_0)$ is an 
analytic path in $T(X_0)$ because $\mu\in ML_{bdd}(X_0)$ (cf. \cite{Sa1}). Let $f_t$ be a 
quasiconformal map from $X_0$ to $X_t$ which is homotopic to $E^{t\mu}$.

For $\alpha\in\mathcal{S}$,  the inequality
$$
l_{f_t(X_0)}(f_t(\alpha ))\leq t i(\mu ,\alpha )+l_{X_0}(\alpha )
$$
implies that
\begin{equation}
\label{eq:up_bd}
\frac{\frac{1}{t}\mathcal{X}([f_t])(\alpha )-i(\mu ,\alpha )}{l_{X_0}(\alpha )}\leq\frac{1}{t}
\end{equation}
for all $\alpha\in\mathcal{S}$ and all $t>0$. 

To obtain the opposite inequality, we choose the universal covering of $X_0$ such that $B(z)=e^{-l_{X_0}(\alpha )}z$ is a cover transformation corresponding to $\alpha$. Let $O$ be the stratum of the lift $\tilde{\mu}$ of $\mu$ to the universal covering $\mathbb{H}$ that contains $e^{l_{X_0}(\alpha )}i$, and let $O_1$ be the stratum of $\tilde{\mu}$ that contains $i$. Normalize the earthquake $E^{t\tilde{\mu}}$ such that $E^{t\tilde{\mu}}|_{O}=id$. Then
$$
B^t=E^{t\tilde{\mu}}|_{O_1}\circ B
$$
is a covering transformation that corresponds to the geodesic on $f_t(X_0)$ homotopic to $f_t(\alpha )$ (cf. \cite{EpMar}). Denote by $l_t$ the translation length of $B_t$ and $l=l_{X_0}(\alpha )$ the translation length of $B$. Let $k_1<0$ and $k_2>0$ be the endpoints of the hyperbolic translation $E^{t\tilde{\mu}}|_{O_1}$, and let $m_t$ be its translation length (cf. Figure 1).


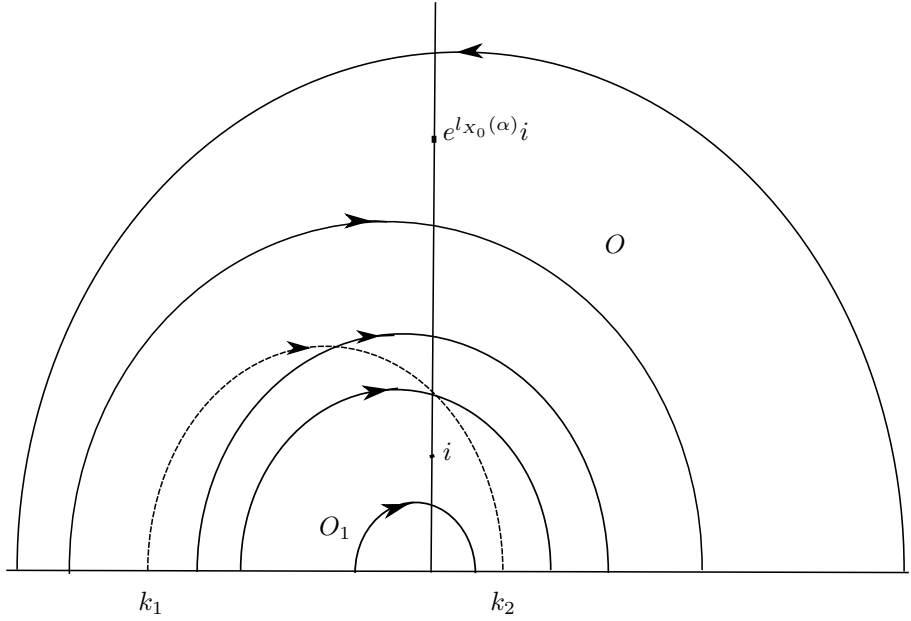
\begin{figure}[h]
\leavevmode \SetLabels
\L(.48*.78) $e^{l_{X_0}(\alpha )}i$\\
\L(.48*.27) $i$\\
\L(.65*.6) $O$\\
\L(.35*.15) $O_1$\\
\L(.16*.03) $k_1$\\
\L(.53*.03) $k_2$\\
\endSetLabels
\begin{center}
\AffixLabels{\centerline{\epsfig{file =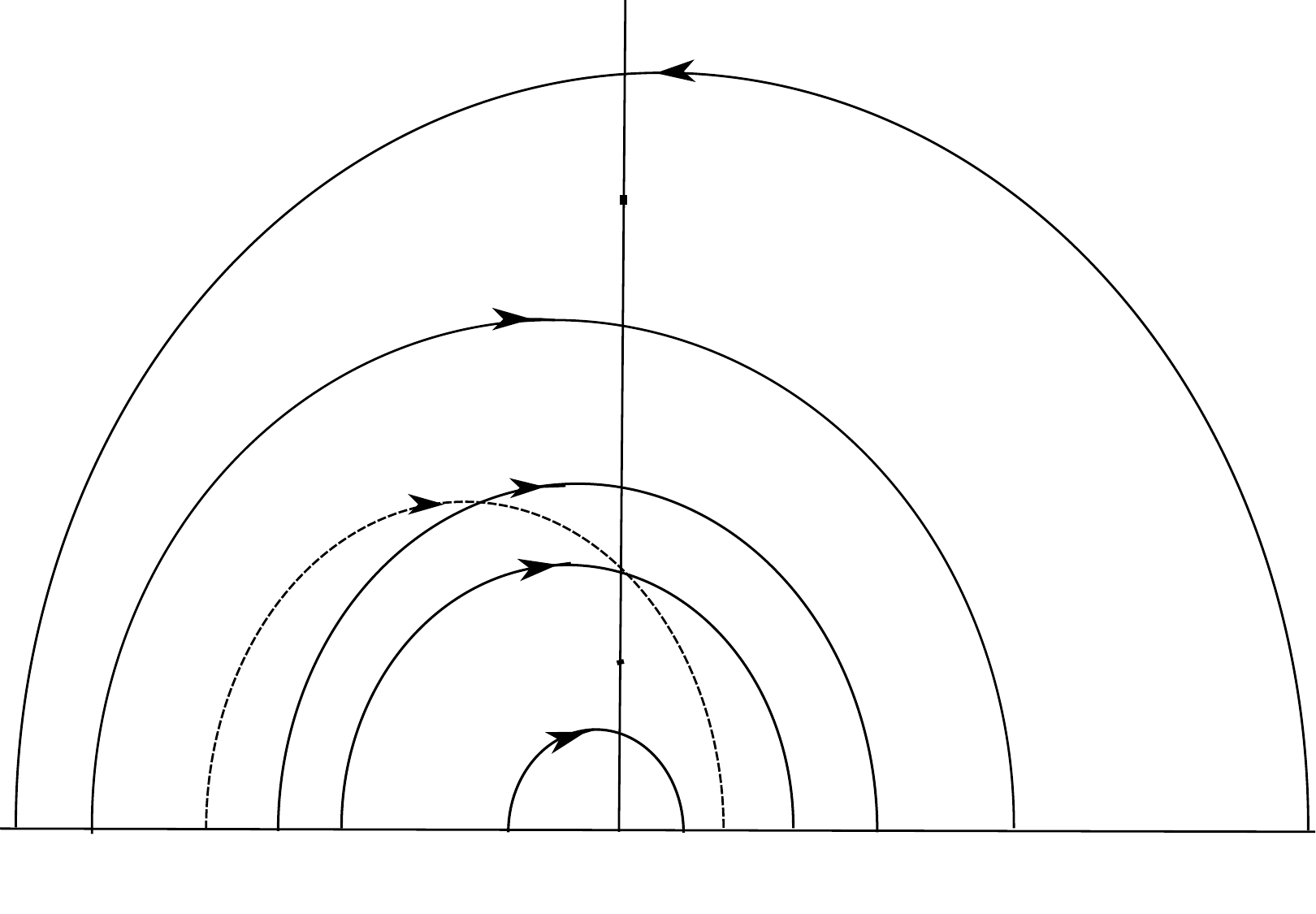,width=12.0cm,angle=0} }}
\vspace{-30pt}
\end{center}
\caption{Computing $E^{t\tilde{\mu}}|_{O_1}$} \label{fig:vertslits}
\end{figure}

 A direct computation (cf. \cite{Sar3}) gives
\begin{equation*}
trace(B^t)=2\cosh \frac{m_t-l}{2}-\frac{2k_1}{k_2-k_1}\Big{(}\cosh\frac{m_t+l}{2}-\cosh\frac{m_t-l}{2}\Big{)}
\end{equation*}
Consequently
$$
2\cosh\frac{l_t}{2}=trace(B^t)\geq 2\cosh \frac{m_t-l}{2}
$$
which implies 
$$
l_t\geq m_t-l.
$$
Since the translation length of a composition of two hyperbolic translations (with non-intersecting axis and translating in the same direction) is at least as large as the sum of their translation lengths (cf. \cite{Th1}), it follows that
$$
m_t\geq t i(\mu ,\alpha).
$$
The above two inequalities give
$$
\frac{1}{t}\frac{l_t}{l}\geq\frac{i(\mu ,\alpha )}{l}-\frac{1}{t}
$$
which implies
\begin{equation}
\label{eq:lw_bd}
\frac{1}{t}\frac{\mathcal{X}([f_t])(\alpha )}{l_{X_0}(\alpha )}- \frac{i(\mu ,\alpha )}{l_{X_0}(\alpha )}\geq-\frac{1}{t}.
\end{equation}
Then equations (\ref{eq:up_bd}) and (\ref{eq:lw_bd}) give that, uniformly in $\alpha\in\mathcal{S}$,
$$
\lim_{t\to\infty}\frac{1}{t}\frac{\mathcal{X}([f_t])(\alpha )}{l_{X_0}(\alpha )}= \frac{i(\mu ,\alpha )}{l_{X_0}(\alpha )}.
$$
We established that each point in $PML_{bdd}(X_0)$ is in  Thurston's boundary.

\vskip .2 cm

Let $\sigma\in \mathbb{R}_{\geq 0}^{\mathcal{S}}$ be such that its projective class $[\sigma ]$ is in length spectrum Thurston's boundary. We need to establish that $[\sigma ]$ is in the closure of $PML_{bdd}(X_0)$ for the normalized supremum norm. 

There exists a sequence $[f_n]\in T(X_0)$ that converges to the projective class $[\sigma ]\in P\mathbb{R}^{\mathcal{S}}_{\geq 0}$. 
Let $t_n\to\infty$ as $n\to\infty$ be such that $\frac{1}{t_n}\mathcal{X}([f_n])\to\sigma$ as $n\to\infty$ in the normalized supremum norm. Necessarily we have $\sup_n\|\frac{1}{t_n}\mathcal{X}([f_n])\|_{\infty}^{norm}<\infty$.

Let $f_n$ be represented by a sequence of earthquakes $E^{t_n'\mu_n}$ with $\|\mu_n\|_{Th}=1$ and $t_n'>0$. Then $t_n'\to\infty$ as $n\to\infty$ and the first part of the proof gives
$$
\|\frac{1}{t_n'}\mathcal{X}([f_n])-\mu_n\|_{\infty}^{norm}<\frac{1}{t_n'}.
$$ 
Note that if $\|\mu_n\|_{Th}=1$ then $\|\mu_n\|_{\infty}^{norm}\leq 2$. Then the above inequality implies that $\|\frac{1}{t_n'}\mathcal{X}([f_n])\|_{\infty}^{norm}\leq 3$ for all $t_n'$ with $n$ large enough and the sequence $\frac{t_n'}{t_n}$ is bounded from the above and below by positive numbers. By choosing a subsequence, if necessary, we can assume that $\frac{t_n'}{t_n}\to c>0$ as $n\to\infty$. It follows that, as $n\to\infty$,
$$
\|\frac{1}{t_n}\mathcal{X}([f_n])-c\mu_n\|_{\infty}^{norm}\to 0
$$
which implies
$$
\| c\mu_n-\sigma\|_{\infty}^{norm}\to 0
$$
and the proof is completed.
\end{proof}

\subsection{Infinite surfaces with bounded geodesic pants decompositions}

We consider a hyperbolic surface $X_0$ which can be decomposed into geodesic pairs of pants with cuffs $\{ \alpha_n\}_{n\in\mathbb{N}}$ such that
$$
1/M\leq l_{X_0}(\alpha_n) \leq M
$$
for some $M>1$ and for all $n\in\mathbb{N}$. We say that such $X_0$ has a {\it bounded geodesic pants decomposition}. The next proposition establishes that length spectrum Thurston's boundary coincides with  Thurston's boundary for $T(X_0)$ introduced using the geodesic currents (cf. \cite{Sar5}).

\begin{proposition}
\label{prop:p2}
Let $X_0$ be a geodesically complete infinite area hyperbolic surface with bounded geodesic pants decomposition. Then length spectrum Thurston's boundary is equal to the space of projective bounded measured laminations $PML_{bdd}(X_0)$ on $X_0$. 
\end{proposition}

\begin{proof}
Consider a sequence of points $[f_k]\in T(X_0)$ that converge to (the projective class of) $L^{*}\in\mathbb{R}^{\mathcal{S}}_{\geq 0}$ 
in length spectrum Thurston's boundary of $T(X_0)$.
Let $t_k\to\infty$ as $k\to\infty
$ be such that $\frac{1}{t_k}\mathcal{X}([f_k])\to L^{*}$ in $\mathbb{R}
_{\geq 0}^{\mathcal{S}}-\{\bar{0}\}$, where $\bar{0}(\alpha )=0$ for all $\alpha\in\mathcal{S}$. Let $E^{t_k\beta_k}$ be a sequence of earthquakes
of $\mathbb{H}$ such that $E^{t_k\beta_k}|_{S^1}=f_k$, where $\|\beta_k\|
_{Th}<\infty$ (cf. \cite{Th1}).

The proof of the above proposition gives
\begin{equation}
\label{eqn:ls-ineq}
\Big{|}\frac{1}{t_k}\frac{\mathcal{X}(E^{t_k\beta_k})(\alpha )}{l_{X_0}(\alpha )}-\frac{i(\beta_k,\alpha )}{l_{X_0}(\alpha )}
\Big{|}\leq \frac{1}{t_k}
\end{equation}
for all $\alpha\in\mathcal{S}$. 

Since $\frac{1}{t_k}\mathcal{X}([f_k])\to L^{*}$, the above inequality implies
$$
\Big{|}\frac{i(\beta_k,\alpha )}{l_{X_0}(\alpha )}-
\frac{L^{*}(\alpha )}{l_{X_0}(\alpha )} \Big{|}\rightarrow 0
$$
as $k\to\infty$ uniformly in $\alpha\in\mathcal{S}$. 
Define
$$
\|\beta\|_{ls}=\sup_{\alpha\in\mathcal{S}}\frac{i(\beta ,\alpha )}{l_{X_0}(\alpha )}
$$
for any $\beta\in ML_{bdd}(X_0)$.
The above convergence gives
$$
\sup_{k\in\mathbb{N}}\|\beta_k\|_{ls}=N<\infty.
$$

We use the assumption that $X_0$ has a bounded geodesic pants decomposition in order to prove that $\|\beta_k\|_{Th}$ is bounded in $k$. Indeed, let $\{\alpha_n\}_{n\in\mathbb{N}}$ be cuffs of a geodesic pants decomposition $\mathcal{P}$ of $X_0$ such that there exists $M>1$ with
$$
\frac{1}{M}\leq l_{X_0}(\alpha_n )\leq M $$
for all $n\in\mathbb{N}$, where $\{\alpha_n\}_n$ are cuffs of $\mathcal{P}$.
Let $P^i$ be a geodesic pair of pants in the above decomposition with the cuffs $\alpha_{i_j}$, for $j=1,2,3$. Assume that $\alpha_{i_j}$, for $j=1,2,3$ are different geodesics of $X_0$. Denote by $P_{i_j}$, $j=1,2,3$, adjacent pair of pants to $P^i$ with common cuff $\alpha_{i_j}$. Then there exists a simple closed geodesic $\alpha_{i_j}^{*}$ in $P_{i_j}\cup P^i$ that intersects $\alpha_{i_j}$ in two points such that $l_{X_0}(\alpha_{i_j}^*)$ is bounded from the above and below by positive constants depending only on $M>0$. The components of $P^i-\cup_{j=1}^3(\alpha_{i_j}\cup\alpha_{i_j}^{*})$ are simply connected for each $i$. If two of $\alpha_{i_j}$, for $j=1,2,3$ is the same geodesic then a similar construction yields $\alpha_{i_j}^{*}$ such that components of $P_i-\cup_{j=1}^3(\alpha_{i_j}\cup\alpha_{i_j}^{*})$ are simply connected and that $l_{X_0}(\alpha_{i_j}^{*})$ is bounded in terms of $M$.


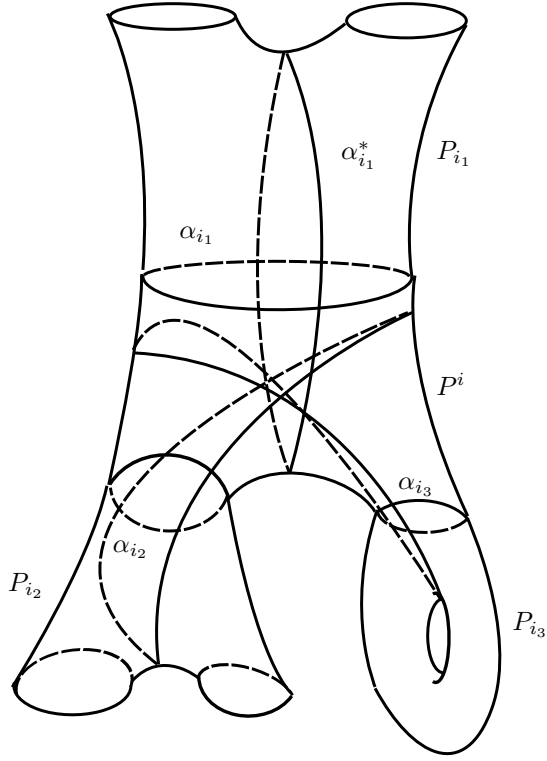
\begin{figure}[h]
\leavevmode \SetLabels
\L(.25*.25) $P_{i_2}$\\
\L(.78*.2) $P_{i_3}$\\
\L(.36*.3) $\alpha_{i_2}$\\
\L(.66*.38) $\alpha_{i_3}$\\
\L(.7*.5) $P^i$\\
\L(.43*.7) $\alpha_{i_1}$\\
\L(.7*.8) $P_{i_1}$\\
\L(.6*.8) $\alpha^{*}_{i_1}$\\
\endSetLabels
\begin{center}
\AffixLabels{\centerline{\epsfig{file =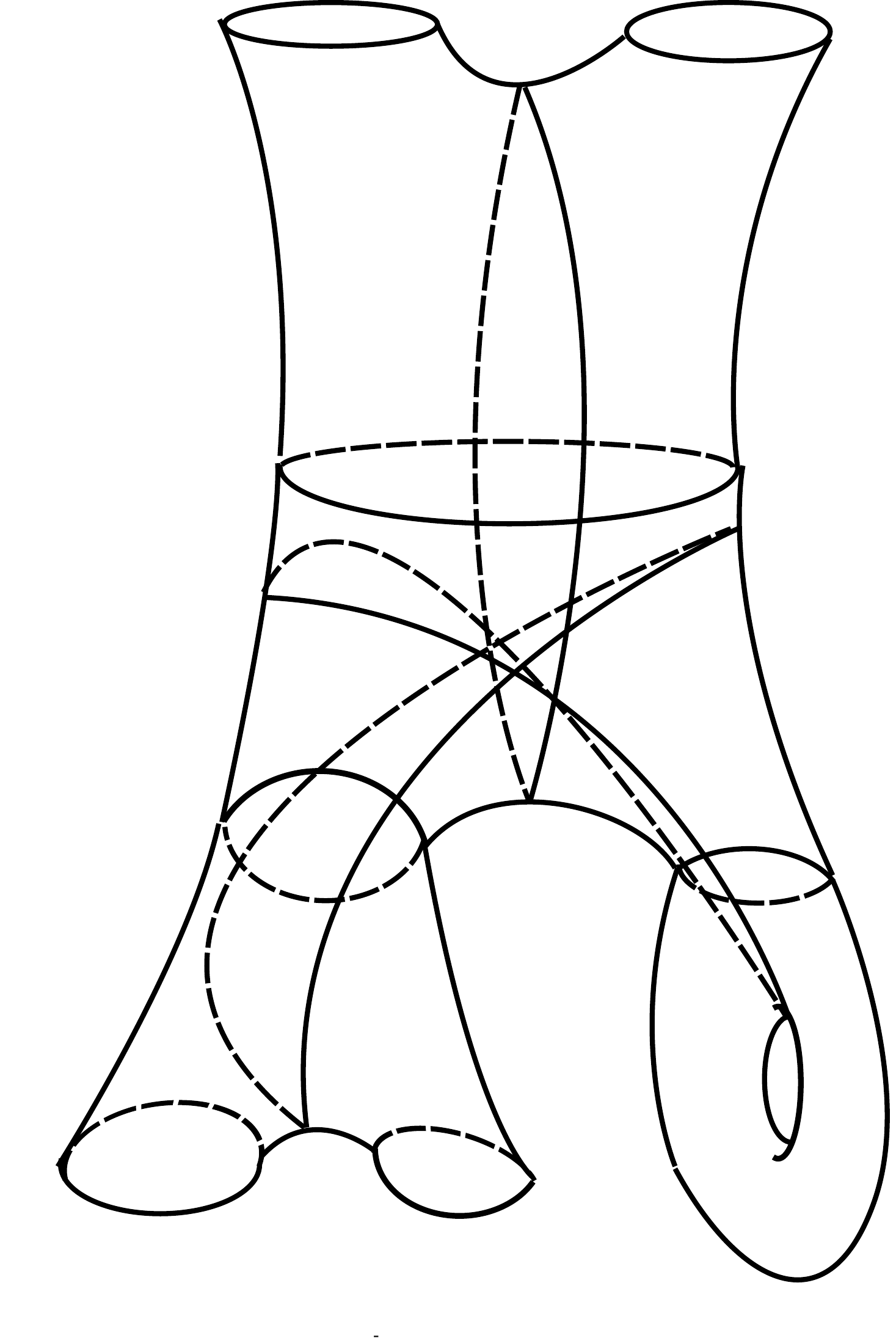,width=7.0cm,angle=0} }}
\vspace{-30pt}
\end{center}
\caption{Decomposition of $X_0$ into bounded polygons.} \label{}
\end{figure}

The above convergence of $\beta_k$ to $L^{*}$ and boundedness of the lengths of $\alpha_{i_j}$ and $\alpha_{i_j}^{*}$ on $X_0$ imply that
$$
i(\beta_k ,\alpha_{i_j}),i(\beta_k ,\alpha_{i_j}^{*})<C(M)
$$
for some constant $C=C(M)$ and for all $i,k\in\mathbb{N}$ and $j=1,2,3$. 
Since $X_0-\cup_i\cup_{j=1}^3\{\alpha_{i_j},\alpha_{i_j}^{*}\}$ has simply connected and uniformly bounded components (that are polygons with at most six sides) whose boundaries are subarcs of $\alpha_{i_j},\alpha_{i_j}^{*}$, we conclude that the supremum over all $k$ and over all above components of the $\beta_k$-mass of the geodesics intersecting components is finite. Since each geodesic arc of length $1$ on $X_0$ can intersect at most finitely many components of $X_0-\cup_i\cup_{j=1}^3\{\alpha_{i_j},\alpha_{i_j}^{*}\}$, it follows that
 $\sup_{k\in\mathbb{N}}\|\beta_k\|_{Th}<\infty$.

By $\sup_{k\in\mathbb{N}}\|\beta_k\|_{Th}<\infty$, there exists a subsequence $\beta_{k_j}$ and $\beta^{*}\in ML_{bdd}(X_0)$ such that $\beta_{k_j}\to\beta^{*}$ as $j\to\infty$ in the weak* topology. (The weak* topology is described in terms of the lifts of the measured laminations $\beta_k$ to the universal covering $\mathbb{H}$.)
Then 
$$
L^{*}(\alpha )=\beta^{*}(\alpha )
$$
for all $\alpha\in\mathcal{S}$ and 
$$
\|\beta^{*}\|_{Th}<\infty.
$$

Thus any point in length spectrum Thurston's boundary is in $PML_{bdd}(X_0)$. The above proposition gives that all points in $PML_{bdd}(X_0)$ are also in length spectrum Thurston's boundary for $T(X_0)$.
\end{proof}

\subsection{Infinite hyperbolic surfaces with upper bounded geodesic pants decompositions}
\label{sec:upper}

Let $X_0$ be a geodesically complete infinite area hyperbolic surface with a geodesic pants decomposition $\mathcal{P}=\{\alpha_n\}_{n\in\mathbb{N}}$ such that $$\sup_n l_{X_0}(\alpha_n)=M<\infty .$$ 
In addition, we assume that there exists a subsequence $\{\alpha_{n_j}\}_j$ with $l_{X_0}(\alpha_{n_j})\to 0$ as $j\to\infty$. Let $P_{n}^1$ and $P_{n}^2$ be the geodesic pairs of pants in $\mathcal{P}$ with a common cuff $\alpha_n$ (possibly $P_n^1=P_n^2$). 
Let $\gamma_n$ be a shortest closed geodesic in $P_{n}^1\cup P_{n}^2$ that intersects $\alpha_n$ in either one point (when $P_n^1=P_n^2$) or in two points (when $P_n^1\neq P_n^2$). We have that (cf. \cite{ALPS})
$$
\frac{l_{X_0}(\gamma_n)}{\max\{ 1,|\log l_{X_0}(\alpha_n)|\}}=O(1) ,
$$
where $O(1)$ is a function pinched between two positive constants.

\begin{proposition}
\label{prop:p3}
Let $X_0$ be a geodesically complete infinite area hyperbolic surface with an upper bounded
geodesic pants decomposition $\mathcal{P}=\{\alpha_n\}_{n\in\mathbb{N}}$ such that a subsequence of cuffs $\alpha_{n_j}$ has lengths going to zero. Then length spectrum Thurston's boundary
of $T(X_0)$ is strictly larger than $PML_{bdd}(X_0)$.
\end{proposition}

\begin{proof}

We use the description of the closure of $T(X_0)$ in the Fenchel-Nielsen coordinates for the pants decomposition $\mathcal{P}=\{ \alpha_n\}_{n\in\mathbb{N}}$. Namely, a marked surface $f:X_0\to X$ is in $T(X_0)$ if and only if its  Fenchel-Nielsen coordinates $\{ (\frac{l_X(\alpha_n)}{l_{X_0}(\alpha_n)},t_X(\alpha_n))\}_{n\in
\mathbb{N}}$ are uniformly bounded;  $f:X_0\to X$ is in the closure of $T(X_0)$ if and only if $\{\frac{l_X(\alpha_n)}{l_{X_0}(\alpha_n)}\}_n$ is bounded and $|t_X(\alpha_n)|=o(\max\{ 1,|\log l_{X_0}(\alpha_n)|\} )$ for all $n$ (cf. \cite{Sar}).

Define a measured lamination $\mu =\sum_j w_j\alpha_{n_j}$ for some $w_j=o(|\log l_{X_0}(\alpha_{n_j})|)$. Then $\mu$ is not Thurston bounded and $E^{t\mu}(X_0)=X^t$ is in the closure of $T(X_0)$ for the length spectrum metric (cf. \cite{Sar}). The proof of Proposition \ref{prop:p1} extends to $\mu$ to get $\frac{1}{t} \mathcal{X}(X^t)\to \mu$ as $t\to\infty$ in the normalised supremum norm. Since each $X^t$ is a limit of points in $T(X_0)$, it follows that $\mu$ is in Thurston's boundary and the proof is competed.
\end{proof}

Theorem 1 from Introduction is established by Propositions \ref{prop:p1}, \ref{prop:p2} and \ref{prop:p3}.

\subsection{Two infinite surfaces with unbounded geodesic pants decompositions}
\label{sec:twosurfaces}

The first surface $X_1$ that we consider is introduced by Kinjo \cite{Kin}. Let 
$\Gamma'$ be the hyperbolic triangle group of signature $(2,4,8)$. Let 
$T'$ be the triangle fundamental polygon for $\Gamma'$ with angles $
\pi /2$, $\pi /4$ and $\pi /8$. Then $\Gamma' (T')$ tiles the hyperbolic 
plane $\mathbb{H}$. Let $T$ be the union of $T'$ and $\gamma'_0 (T')$, where $\gamma'_0\in\Gamma'$ is a reflection in the geodesic containing the side of $T'$ which subtends the angles $\pi /2$ and $\pi /8$ of $T'$. Denote the vertices of $T$ by $a$, $b$ and $c$; the vertex $b$ is 
where $T'$ has angle $\pi /8$ (cf. \cite[Figure 2]{Kin}). We choose three points $a'$, $b'$ and $c'$ close to $a$, $b$ and $c$, 
respectively, in the interior of the triangle $T$ such that $b'$ is on the side of $T'$ containing $b$. The surface $X_1$ is obtained by 
puncturing the hyperbolic plane at the points $\Gamma'\{ a',b',c'\}$ (cf. \cite[Figures 2,3]{Kin}). Kinjo \cite{Kin} proved that the Teichm\"uller 
space $T(X_1)$ is complete in the length spectrum metric.

Let $\{\gamma_i\}_{i=1,\ldots ,8}$ be the elements of $\Gamma'$ that 
fix $a$. Let $l_a$ be the simple closed geodesic which separates the 
eight points $\{\gamma_i(a)\}_{i=1,\ldots ,8}$ from the other punctures 
of $X_1$. We similarly define curves $l_b$ and $l_c$, and then extend 
the definition using $\Gamma'$ to all other groups of eight cusps. The lengths of all $\Gamma'(l_a)$ are the same, as well as the 
lengths of all $\Gamma'(l_b)$, as well as the lengths of all $\Gamma'(l_c)$.

For the triangle $T$, we denote by $l_{a',b'}$ the simple closed geodesic which is homotopic to a simple closed curve in $T$ that separates $a',b'$ from $c'$. We similarly extend the definition to $l_{b',c'}$ and $l_{c',a'}$, and then extend it to all triangles using the invariance under $\Gamma'$. Note that the lengths of $\Gamma'(l_{a',b'})$ are the same, as well as the 
lengths of all $\Gamma'(l_{b',c'})$, and the lengths of all $\Gamma'(l_{c',a'})$.

The lengths of the family of geodesics $\Gamma'(l_a)\cup\Gamma'(l_b)\cup\Gamma'(l_c)\cup
\Gamma'(l_{a',b'})\cup\Gamma'(l_{b',c'})\cup\Gamma'(l_{c',a'})$ are bounded from the below and from the above, and this family separates the surface $X_1$ into finite bounded polygons with uniformly bounded number of sides. Then the proof of Proposition \ref{prop:p3} extends to show that length spectrum Thurston's boundary coincides with $PML_{bdd}(X_1)$.

\vskip .2 cm

Denote by $X_2$ an infinite hyperbolic surface defined by Shiga \cite{Shi} that has geodesic pants decomposition with cuff lengths converging to infinity. The surface $X_2$ contains a sequence $\gamma_n$ of simple closed geodesics with $l_{X_2}(\gamma_n)\to\infty$ as $n\to\infty$ such that for each closed geodesic $\delta$ we have 
\begin{equation}
\label{eqn:lbound}
l_{X_2}(\delta )\geq \sum_{k=1}^{\infty}kl_{X_2}(\gamma_k)i(\gamma_k,\delta ),
\end{equation}
where only finitely many terms are non-zero.
Shiga \cite{Shi} proved that a sequence of full Dehn twists $f_n$ around the curve $\gamma_n$ diverges in the Teichm\"uller metric and it converges to the identity in the length spectrum metric. Thus the two metrics produce different topologies on $T(X_2)$. 

We define $\beta_n$ to be a measured lamination whose support is $\{\gamma_k\}_{k=1,\ldots ,n}$ such that, for $k=1,\ldots ,n$,
$$
\beta_n|_{\gamma_k}=l_{X_2}(\gamma_k).
$$
The projective class $[\beta_n]$ is in $PML_{bdd}(X_2)$. Define
$\beta_{*}$ to be a measured lamination on $X_2$ whose support is $\{\gamma_k\}_{k=1}^{\infty}$ such that,
 for all $k=1,2,\ldots $,
 $$
\beta_*|_{\gamma_k}=l_{X_2}(\gamma_k).
$$
It is clear that the projective class $[\beta_*]$ is not in $PML_{bdd}(X_2)$. 

We prove that $[\beta_n]\to [\beta_*]$ as $n\to\infty$ in the normalized supremum norm. Indeed, let $\delta$ be a simple closed geodesic in $X_2$. 
Then
$$
\frac{|i(\beta_n,\delta)-i(\beta_*,\delta )|}{l_{X_2}(\delta )}=\sum_{k=n+1}^{\infty}
\frac{i(\beta_k,\delta )}{l_{X_2}(\delta )}=\frac{\sum_{k=n+1}^{\infty}
i(\delta ,\gamma_k)l_{X_2}(\gamma_k )}{\sum_{k=1}^{\infty}ki(\delta ,\gamma_k)l_{X_2}(\gamma_k )}\leq\frac{1}{n+1}
$$
and $[\beta_*]$ is in length spectrum Thurston's boundary of $T(X_2)$. Therefore the boundary is larger than $PML_{bdd}(X_2)$. 

\vskip .2 cm

\noindent {\bf Open problem:} 
 Assume that a sequence in $T(X_0)$ converges to a bounded projective measured lamination in length spectrum Thurston's boundary. Is it true that the sequence converges in Thurston's boundary introduced using geodesic currents?


\begin{thebibliography}{Thua}

\vskip .5cm

\bibitem{ALPS} D. Alessandrini, L. Liu, A. Papadopoulos, W. Su and Z. Sun,  {\it On Fenchel-Nielsen coordinates on Teichm\"uller spaces of surfaces of infinite type}, Ann. Acad. Sci. Fenn. Math.  36  (2011),  no. 2, 621-659.

\bibitem{ALPS1}  D. Alessandrini, L. Liu, A. Papadopoulos and W. Su, {\it On the inclusion of the quasiconformal Teichm\"uller space into the length-spectrum Teichm\"uller space}, preprint, arXiv:1201.6030.

\bibitem{BK} A. Basmajian and Y. Kim, {\it Geometrically infinite surfaces with discrete length spectra}, Geom. Dedicata  137  (2008), 219-240. 

\bibitem{BasSar} A. Basmajian and D. \v Sari\' c, {\it Geodesically Complete Hyperbolic Structures}, preprint.

\bibitem{Bear} A. Beardon, {\it The geometry of discrete groups},  
Graduate Texts in Mathematics, 91. Springer-Verlag, New York, 1983.


\bibitem{Bo} F. Bonahon, {\it The geometry of Teichm\"uller space via geodesic currents}, Invent. Math.  92  (1988),  no. 1, 139-162.

\bibitem{Bus} P. Buser, {\it Geometry and spectra of compact Riemann surfaces}, Reprint of the 1992 edition. Modern Birkh\"auser Classics. Birkh\"auser Boston, Inc., Boston, MA, 2010.

\bibitem{EpMar} D. B. A. Epstein and A. Marden, {\it Convex hulls in hyperbolic space,
a theorem of Sullivan and measured pleated surfaces}, LMS Lecture
Notes 111, pages 112-253, Cambridge University Press, 1987.

\bibitem{EpMarMar} D. B. A. Epstein, A. Marden ad V. Markovic, {\it Quasiconformal homeomorphisms and the convex hull boundary}, Ann. of Math. (2) 159 (2004), no. 1, 305-336.

\bibitem{FLP} A. Fathi, F. Laudenbach and V. Po\'enaru, {\it Thurston's work on surfaces},
Translated from the 1979 French original by Djun M. Kim and Dan Margalit. Mathematical Notes, 48. Princeton University Press, Princeton, NJ, 2012.

\bibitem{GL}  F. Gardiner and N. Lakic,  {\it Quasiconformal Teichm\"{u}ller Theory},
Mathematical Surveys and Monographs, Volume 76, A.M.S. 2000.

\bibitem{GHL} F. Gardiner, J. Hu and N. Lakic, {\it
Earthquake curves}, Complex manifolds and hyperbolic geometry (Guanajuato, 2001),  141-195, 
Contemp. Math., 311, Amer. Math. Soc., Providence, RI, 2002. 

\bibitem{Kin} E. Kinjo, {\it On Teichm\"uller metric and the length spectrums of topologically infinite Riemann surfaces}, 
Kodai Math. J.  34  (2011),  no. 2, 179-190.

\bibitem{MSS} R. Ma\~{n}\'{e}, P. Sad and D. Sullivan, {\it On the dynamics of rational maps},
Ann. Sci. Ecole Norm. Sup, 16, 193-217, 1983.



\bibitem{Matz} K. Matsuzaki, {\it A classification of the modular transformations of infinite dimensional Teichm\"uller spaces}, In the tradition of Ahlfors-Bers. IV,  167-177, Contemp. Math., 432, Amer. Math. Soc., Providence, RI, 2007. 

\bibitem{MS} H. Miyachi and D. \v Sari\' c, {\it  Uniform weak* topology and earthquakes in the hyperbolic plane}, Proc. Lond. Math. Soc. (3)  105  (2012),  no. 6, 1123Ð1148.

\bibitem{Ot} J. P. Otal, {\it About the embedding of Teichm\"uller space in the space of geodesic H\"older distributions},  Handbook of TeichmŸller theory. Vol. I,  223-248, IRMA Lect. Math. Theor. Phys., 11, Eur. Math. Soc., Z\"urich, 2007. 

\bibitem{Sa1} D. \v Sari\' c, {\it Real and Complex Earthquakes}, Trans.
Amer. Math. Soc. 358 (2006), no. 1, 233-249.

\bibitem{Sa2} D. \v Sari\' c, {\it Bounded earthquakes}, Proc. Amer. Math. Soc.  136  (2008),  no. 3, 889-897. 

\bibitem{Sar} D. \v Sari\' c, {\it Fenchel-Nielsen coordinates on upper bounded pants decompositions}, to appear Math. Proc. Cambridge Philos. Soc., arXiv:1209.5819. 

\bibitem{Sa} D. \v Sari\' c, {\it Geodesic currents and Teichm\"uller spaces}, Topology  44  (2005),  no. 1, 99-130. 

\bibitem{Sar3} D. \v Sari\' c, {\it Earthquakes in the length spectrum Teichm\"uller space}, Proc. Amer. Math. Soc. 143 (2015), no. 4, 1531-1543. 

\bibitem{Sar5} D. \v Sari\' c, {\it Thurston's boundary for Teichm\"uller spaces of infinite surfaces: geodesic currents}, preprint, available on arxiv.org.


\bibitem{Shi} H. Shiga, {\it On a distance defined by the length spectrum of Teichm\"uller space}, Ann. Acad. Sci. Fenn. Math.  28  (2003),  no. 2, 315Ð326.

\bibitem{Th1} W. Thurston, {\it On the geometry and dynamics of diffeomorphisms of surfaces}, Bull. Amer. Math. Soc. (N.S.)  19  (1988),  no. 2, 417-431.

\bibitem{Th} W. Thurston, {\it  Earthquakes in two-dimensional hyperbolic geometry},  Low-dimensional topology and Kleinian groups (Coventry/Durham, 1984),  91-112, London Math. Soc. Lecture Note Ser., 112, Cambridge Univ. Press, Cambridge, 1986. 

\bibitem{Wol} S. Wolpert, {\it The Fenchel-Nielsen deformation}, Ann. of Math. (2)  115  (1982), no. 3, 501-528.

\end{thebibliography}
\end{document}